\documentclass[11pt]{article}
\usepackage{fullpage} 
\usepackage{blkarray}
\usepackage{amssymb}

\usepackage{amsmath}
\usepackage{amsthm}
\usepackage{mathtools}
\usepackage{latexsym}
\usepackage[utf8]{inputenc}
\usepackage[english]{babel}
\usepackage{graphicx}
\usepackage{wrapfig}
\usepackage{caption}
\usepackage{subcaption}
\usepackage{txfonts}
\usepackage{lmodern}
\usepackage{mathrsfs}
\usepackage{algorithm}
\usepackage{dsfont}
\usepackage{enumerate}
\usepackage[hidelinks]{hyperref}
\usepackage{multicol}
\usepackage{csquotes}
\usepackage{stmaryrd}
\usepackage{tikz}
\usepackage{comment}
\usepackage{enumitem}
\usepackage{longtable}
\usepackage{bbm}
\usepackage{cleveref}

\usepackage{skak}

\newtheorem{theo}{Theorem}

\newtheorem{lemma}[theo]{Lemma}

\numberwithin{theo}{section}

\newcommand{\NN}{\mathbb{N}}
\newcommand{\ZZ}{\mathbb{Z}}
 
\newcommand{\RR}{\mathbb{R}}

\newcommand{\cT}{\mathcal{T}}

\DeclareMathOperator{\lcm}{lcm}

\def\P{\mathcal{P}}
\def\C{\mathcal{C}}

\newcommand\restr[2]{{
		\left.\kern-\nulldelimiterspace 
		#1 
		\vphantom{\big|} 
		\right|_{#2} 
}}

\title{Crossing Numbers of Billiard Curves in the Multidimensional Box via Translation Surfaces}

\begin{document}

\author{%
Felix Christian Clemen\footnote {Department of Mathematics, Karlsruhe Institute of Technology, 76131 Karlsruhe, Germany, E-mail: \texttt{felix.clemen@kit.edu}.} 
 \and Peter Kaiser\footnote {Department of Mathematics, Karlsruhe Institute of Technology, 76131 Karlsruhe, Germany, E-mail: \texttt{peter.kaiser2@kit.edu}.} 
}

\maketitle
\begin{abstract}\noindent
The billiard table is modeled as an $n$-dimensional box $[0,a_1]\times [0,a_2]\times \ldots \times [0,a_n] \subset \RR^n$, with each side having real-valued lengths $a_i$ that are pairwise commensurable. A ball is launched from the origin in direction $d=(1,1,\ldots,1)$. The ball is reflected if it hits the boundary of the billiard table. It comes to a halt when reaching a corner. We show that the number of intersections of the billiard curve at any given point on the table is either $0$ or a power of $2$. To prove this, we use algebraic and number theoretic tools to establish a bijection between the number of intersections of the billiard curve and the number of satisfying assignments of a specific constraint satisfaction problem.
\end{abstract}

\section{Introduction}
The trajectory of a ball on a billiard table is studied in various areas of mathematics, for example in ergodic theory, Teichmüller theory~\cite{McMullen, Veech1989}, from a geometry perspective~\cite{taba,Veech1992}, or through a combinatorial lens in the setting of arithmetic billiards~\cite{BilliardNdim, veech, Yangcheng, ABilliards}.\par\noindent
We model the \emph{billiard curve} as follows. Let $\mathcal{T} \subset \RR^n$ be a domain and $p\in \mathcal{T}$ a point inside this domain. Additionally, let $d \in \RR^n$ be a direction. The billiard curve is the trajectory starting at $p$ moving in direction $d$, where the direction is reflected if the trajectory reaches the boundary, in the sense of an elastic collision. If it reaches a corner, the motion ceases. A natural question is to determine if for a given domain there exists a periodic trajectory. 
For a specific domain, it might be challenging to decide whether a trajectory is periodic, stops, or has infinite length \cite{Birkhoff}. We are considering a setting where all trajectories have finite length.
\par\noindent  

A set of positive reals $a_1, a_2,\ldots, a_n\in \mathbb{R}$ is called \emph{pairwise commensurable} if all ratios $a_i/a_j$ are rational. In this paper, the domain $\mathcal{T}$ is always the box $\mathcal{T}=[0,a_1]\times [0,a_2]\times \ldots \times [0,a_n]$ with pairwise commensurable side lengths $a_i$, the starting point is the origin, i.e. $p=(0,\ldots,0)$, and the direction $d=(1,\ldots,1)$. We denote by $[n]=\{1,2,\ldots,n\}$ the set of the first $n$ integers. A \emph{corner} is a point $w\in \cT$ with $w_i\in \{0,a_i\}$ for every $i\in[n]$. In this setting, the trajectory always stops at a corner different from the origin.\par\noindent

We ask how often the trajectory crosses a given point within the domain. Given a point $v\in \cT$ let the \emph{crossing number of} $v$, denoted by $m_{\cT}(v)$, be the number of times the trajectory crosses $v$. In the 2-dimensional scenario, i.e. $\mathcal{T}=[0,a_1]\times [0,a_2]$, for every $v\in \mathcal{T}$, trivially $m_{\cT}(v)\in\{0,1,2\}$ simply because the trajectory is the union of lines with only two directions. \par

Our results concern the multi-dimensional case. For any positive real numbers $a_1,\ldots,a_n$ and $\mathcal{T}=[0,a_1]\times [0,a_2]\times \ldots \times [0,a_n]$, we have $m_{\cT}(v)\leq 2^{n-1}$ for a given point $v\in \cT$ because $v$ can only be crossed in one of the two directions $d$ or $-d$ for a given $d\in\{-1,1\}^n$. We give a more precise description of crossing numbers.

\begin{theo}\label{Thm:Main}
Let $T=[0,a_1]\times [0,a_2]\times \ldots \times [0,a_n]$ be the box with pairwise commensurable side lengths $a_i>0$. Then for every $v\in \mathcal{T}$, we have $m_{\cT}(v)\in\{0\}\cup \{2^k: k\geq 0\}$.
\end{theo}

For our proof, we initially establish a discrete structure, so that we can analyze the problem by using algebraic and number theoretic methods. In the second step, we formalize the setting in the language of translation surfaces. Following this, we establish a bijection to the number of solutions of a specific boolean constraint satisfaction problem. Finally, we count those solutions.\par 
In the case of pairwise coprime side lengths, we determine the crossing number of a given $v \in \cT$ precisely. 

\begin{theo}
 \label{Thm:Main2} 
 Let $a_1,\ldots,a_n$ be pairwise coprime positive integers, and let $v\in \mathcal{T} \cap \NN^n$ not be a corner. If $v_i \equiv 0 \bmod 2$ for all $i \in [n]$ or $v_i \equiv 1 \bmod 2$ for all $i \in [n]$, then 
 \begin{align*}
 m_{\cT}(v)=2^{n-1-|\{i: \ v_i=0 \text{ or } v_i=a_i  \}| }.
 \end{align*}
 Otherwise,  $m_{\cT}(v)=0$. 
\end{theo}
In the context of arithmetic billiards, the number of self-intersecting points in the two-dimensional case has been explored by Perucca, Reguengo De Sousa and Tronto~\cite{ABilliards}, as well as by Da Veiga Bruno~\cite{BilliardNdim}. Another counting question asks for the number of bouncing points, i.e. the number of points in which the trajectory meets the boundary of the domain, this has been studied by Steinhaus \cite{Steinhaus}, Gardner \cite{Gardner} and also in the multi-dimensional case by Li~\cite{Yangcheng}.\par 

The paper is organized as follows. In the first step, in \Cref{sec:DiscreteStructure}, we discretize the problem. As a warm-up, we will present an elementary proof of Theorem~\ref{Thm:Main2} in Section~\ref{sec:Warmup}. In Section~\ref{sec:TransSur} we introduce translation surfaces. Finally, in Section~\ref{sec:mainProof}, we present the proof of our main result, Theorem~\ref{Thm:Main}.

\section{Discretization of the billiard table}\label{sec:DiscreteStructure}
For our further reasoning, we rely on a discretization of the problem. Our first observation is that scaling the billiard table $\cT$ does not change the form of the billiard curve. Since the side lengths $a_i$ are pairwise commensurable, after scaling, we can assume that $a_i\in \NN$ for all $i\in [n]$. Throughout this paper, from now on, all $a_i$ will be assumed to be integers.\par\noindent
Observe that the set of \emph{intersection points}, i.e. the points satisfying $m_{\mathcal{T}}(v)\geq 2$, is a subset of $\NN^n \cap \cT$. This means, it suffices to prove \Cref{Thm:Main} and \Cref{Thm:Main2} for $v\in \NN^n \cap \cT$. This is crucial for our argumentation since we will heavily rely on number theoretic machinery. The following lemma which follows by a result from Li~\cite{Yangcheng} gives a necessary condition for a point $v\in \NN^n \cap \cT$ to be on the billiard curve. We include a proof for completeness. 
\begin{lemma}[Theorem 4.3 in Li~\cite{Yangcheng}]
\label{lem:Condition}
    Let $v \in \cT \cap \NN^n$ with $m_{\cT}(v)>0$. Then $v_i \equiv 0 \bmod 2$ for all $i \in [n]$ or $v_i \equiv 1 \bmod 2$ for all $i \in [n]$.
    \begin{proof}
       The starting point $(0,\ldots,0)$ satisfies $v_i \equiv 0 \bmod 2$ trivially. The current direction $d$ of the trajectory is always of the form $d \in \{-1,1\}^n$. For each $v \in \cT \cap \NN^n$ with $v + d \in \cT$ and $v_i \equiv v_j \bmod 2$ for all $i,j\in [n]$, we have $v + d \in \cT \cap \NN^n$ and $v_i+d_i \equiv v_j+d_j \bmod 2$ since the parity in all entries is changed. The claim follows inductively.
    \end{proof}
\end{lemma}

\noindent The following lemma determines the length of the trajectory, i.e. the sum of the lengths of the linear segments of the trajectory. Recall that our trajectory starts in the origin and stops in a corner.

\begin{lemma}[see e.g. Theorem 4.8 in Li~\cite{Yangcheng}]
\label{lem:LengthTrajectory}
    Let $a_i \in \NN$ for all $i\in [n]$. Then the trajectory has a length of $\sqrt{n} \ell$, where $\ell = \lcm(a_1,\ldots,a_n)$.
    \begin{proof}
    Whenever the initial segment of the trajectory has length a multiple of $a_i \sqrt{n}$, the current point has $a_i$ or $0$ as the $i$-th coordinate. Since every corner $p$ satisfies $p_i\equiv 0 \bmod a_i$ for every $i\in [n]$, the length of the trajectory is $\ell \sqrt{n}$. 
    \end{proof}
\end{lemma}

\noindent An advantage of the discretization of the billiard table is that we can use the Chinese Remainder Theorem.
\begin{theo}[Chinese Remainder Theorem]\label{Thm:ChineseRemainder}
    Let $a_1, \ldots ,a_n$ be integers greater than $1$ and $\ell=\lcm(a_1,  a_2, \ldots,a_n)$. Further, let $v_1, \ldots , v_n$ be integers such that $0 \leq v_i < a_i$ for every $i \in [n]$. Iff for all $i\neq j$ we have $v_i\equiv v_j \bmod gcd(a_i,a_j)$, then there is a unique integer $x$, such that $0 \leq x < \ell$ and $x\equiv v_i \bmod a_i$. 
\end{theo}

\section{The coprime case: Proof of Theorem~\ref{Thm:Main2}}\label{sec:Warmup}

Let $a_1,a_2,\ldots,a_n$ be pairwise coprime integers.
Given $v\in \cT\cap \mathbb{N}^n$, define 
$$I_1=\{i\in [n]: v_i=0\}, \quad I_2=\{i\in [n]: v_i=a_i\} \quad \text{and} \quad I(v)=I_1\cup I_2.$$ 
First, we present an upper bound on $m_{\cT}(v)$ depending on $|I(v)|$.
\begin{lemma}
\label{mvup}
Let $v\in \cT\cap \mathbb{N}^n$ be a non-corner. Then  $m_{\cT}(v)\leq 2^{n-1-|I(v)|}$.
\end{lemma}
\begin{proof}
Let $W_v \subseteq \{-1,1\}^n$ be the set of directions which from $v$ point towards the board, i.e. the set of vectors $w\in \{-1,1\}^n$ such that $v+\varepsilon w\in \cT$ for some $\varepsilon>0$. Note that every $w\in W_v$ satisfies $w_i=1$ if $i\in I_1$ and $w_i=-1$ if $i\in I_2$. Thus, $|W_v|=2^{n-|I(v)|}$.
Since the trajectory enters and leaves the point $v$, we conclude $m_\cT(v)\leq |W_v|/2\leq 2^{n-|I(v)|-1} $ for every $v\in \cT\cap \mathbb{N}^n$.
\end{proof}
Recall that, by definition, $m_\cT(v)=1$ for $v$ being the start corner or the end corner. For all other corners $v\in \cT$, we have $m_\cT(v)=0$. A point $v\in \cT \cap \NN^n$ is called \emph{reachable} if
$v_i \equiv 0 \bmod 2$ for all $i \in [n]$ or $v_i \equiv 1 \bmod 2$ for all $i \in [n]$. Set $b_k$ to be the number of reachable points $v\in \cT \cap \NN^n$ such that $|I(v)|=k$. Observe that the contraposition of Lemma~\ref{lem:Condition} states that if $v$ is not reachable, then by $m_{\cT}(v)=0$.

\begin{lemma}
\label{bk}
Let $0\leq k\leq n$. Then
\begin{align*}
b_k=2^{1-n+k} \sum_{\substack{J\subseteq[n]: \\ |J|=k }}  \ \prod_{\substack{j\in [n]\setminus J }} (a_j-1).
\end{align*}
\end{lemma}
\begin{proof}
We distinguish two cases.

\noindent
\textbf{Case 1:} $2 \nmid a_i$ for every $i\in [n]$.

Let $v \in \cT \cap \NN^n$ such that $|I(v)|=k$ and either $v_i \equiv 0 \bmod 2$ for all $i\in [n]$ or $v_i \equiv 1 \bmod 2$ for all $i\in [n]$. Because $|I(v)|=k$ there are exactly $k$ coordinates of $v$ which are either $0$ or $a_i$. Further, since $a_i \equiv 1 \bmod 2$ for all $i \in [n]$, these $k$ coordinates are simultaneously either $0$ or $a_i$. The $n-k$ coordinates, which are not on the border, must have the same parity as the coordinates on the border. Therefore, they can obtain precisely $\frac{a_i-1}{2}$ possible values. This yields, for $0\leq k \leq n$,
\begin{align*}
b_k=\sum_{\substack{J\subseteq[n]: \\ |J|=k }} \ 2 \ \prod_{\substack{j\in [n]\setminus J }} \frac{a_j-1}{2}=2^{1-n+k} \sum_{\substack{J\subseteq[n]: \\ |J|=k }}  \ \prod_{\substack{j\in [n]\setminus J }} (a_j-1).
\end{align*}
\textbf{Case 2:} $2 \mid a_i$ for some $i\in [n]$.

Note that there is at most one even $a_i$ since the $a_i$ are assumed to be pairwise coprime. Without loss of generality, we can assume $2\mid a_1$. We count the number of reachable $v \in \cT \cap \NN^n$ such that $|I(v)|=k$ by looking at the cases $v_1\in \{0, a_1\}$ and $v_1\notin \{0, a_1\}$ separately. 

If the first coordinate is on the border, i.e. $v_1\in \{0, a_1\}$, then $v_i\equiv 0 \bmod 2$ for all $i\in[n]$. Thus, $v_j=0$ for $j\in I(v)$. There are two choices for $v_1$. The coordinates not on the border must be even, and therefore can obtain $\frac{a_j-1}{2}$ possible values. We conclude that the number of reachable $v \in \cT \cap \NN^n$ such that $|I(v)|=k$ and $v_1\in \{0, a_1\}$ is  
\begin{align}
\label{eqcount1}
\sum_{\substack{J\subseteq[n]: \\ |J|=k,\ 1\in J }} \left(\ 2  \ \prod_{j\in [n]\setminus J} \frac{a_j-1}{2} \right).
\end{align}
If the first coordinate is not on the border, i.e. $v_1\notin \{0, a_1\}$, there are $a_1-1$ choices for $v_1$. Since $v$ is reachable, this choice determines the coordinates $v_j$ for $j\in I(v)$ uniquely, and also the parity of $v_j$ for $j\notin I(v)$. There are $\frac{a_j-1}{2}$ possible choices for non-boundary coordinates. Thus, the number of reachable $v \in \cT \cap \NN^n$ such that $|I(v)|=k$ and $v_1\notin \{0, a_1\}$ is  
\begin{align}
\label{eqcount2}
\sum_{\substack{J\subseteq[n]: \\ |J|=k,\ 1\not\in J }}  \left(\   (a_1-1)\ \prod_{\substack{j\in [n]\setminus J \\ j\neq 1 }}  \frac{a_j-1}{2} \right).
\end{align}
By adding up \eqref{eqcount1} and \eqref{eqcount2} we obtain, for  $0\leq k \leq n$,
\begin{align*}
b_k&=\sum_{\substack{J\subseteq[n]: \\ |J|=k,\ 1\in J }} \left(\ 2  \ \prod_{j\in [n]\setminus J} \frac{a_j-1}{2} \right)+\sum_{\substack{J\subseteq[n]: \\ |J|=k,\ 1\not\in J }}  \left(\   (a_1-1)\ \prod_{\substack{j\in [n]\setminus J \\ j\neq 1 }}  \frac{a_j-1}{2} \right)\\
&=2^{1-n+k} \sum_{\substack{J\subseteq[n]: \\ |J|=k }}  \ \prod_{\substack{j\in [n]\setminus J }} (a_j-1),
\end{align*}
completing the proof of this lemma.
\end{proof} 

\begin{proof}[Proof of Theorem~\ref{Thm:Main2}]
The length of the trajectory is $\sqrt{n} \ell$ by \Cref{lem:LengthTrajectory}, with $\ell=\lcm(a_1,\ldots,a_n)$. Therefore, the trajectory passes through $\ell$ points from $\cT \cap \NN^n$, when we count with multiplicities but without the starting point. The length $\ell$ can be expressed as the sum over the crossed points in $\cT \cap \NN^n$ times their multiplicity, i.e.
\begin{align}
\label{laenge1}
\ell&=\sum_{v\in \cT \cap \NN^n}m_{\cT}(v).
\end{align}
Further, by partitioning the points by the dimension of the border they are lying on and by Lemma~\ref{lem:Condition}, 
\begin{align}
\label{laenge}
\ell&=\sum_{v\in \cT \cap \NN^n}m_{\cT}(v)=\sum_{k=0}^{n}\ \sum_{\substack{v\in \cT\cap \NN^n:\\  |I(v)|=k\\ v \text{ is reachable}}}m_{\cT}(v).
\end{align}
By combining Lemmas~\ref{mvup}, \ref{bk} with  \eqref{laenge}, we obtain
\begin{align}
\label{laenge2}
 \ell &= \sum_{k=0}^{n}\ \sum_{\substack{v\in \cT\cap \NN^n:\\  |I(v)|=k \\ v \text{ is reachable}}}m_{\cT}(v) \leq \sum_{k=0}^{n}\ b_k 2^{n-1-k} =  \sum_{k=0}^{n}\ \sum_{\substack{J\subseteq[n]: \\ |J|=k }}  \ \prod_{\substack{j\in [n]\setminus J }} (a_j-1) \\
 &= \sum_{\substack{J\subseteq[n]}}  \ \prod_{\substack{j\in J }} (a_j-1)=\prod_{i\in [n]}a_i = \ell.
 \label{laenge3}
\end{align}
The second to last equality in \eqref{laenge3} holds by a standard inclusion-exclusion argument, or alternatively can be proven via induction.
We conclude that the inequality in \eqref{laenge2} must hold with equality. Thus, if $v\in \cT \cap \NN^n$ is reachable, then $m_{\cT}(v)=2^{n-1-|I(v)| }$. Recall that if $v$ is not reachable, then by Lemma~\ref{lem:Condition}, $m_{\cT}(v)=0$. This completes the proof of Theorem~\ref{Thm:Main2}.
\end{proof}

\section{Translation surfaces}\label{sec:TransSur}

A problem with analyzing the trajectory is that it changes its direction whenever it hits the boundary. We will overcome this issue by artificially extending the board in such a way that the trajectory on the new board is simply a single straight line. This will be done in two steps. In the first step, we extend the domain to $2\cT$. On this new domain the trajectory will not change its directions, but will not be connected in the Euclidean sense as a trade-off. In the second step, we embed $2\cT$ into $\RR^n$. On this domain the trajectory will be a straight line.\par\medskip
 We now consider $2\cT=[0,2a_1] \times\ldots\times[0,2a_n]$ instead of $\cT$. Let $e_i\in \mathbb{R}^n$ be the standard basis vector with entry $1$ in the $i$-th component. The domain $\cT$ is bounded by the hyperplanes 
$$P_i:= \mathrm{span}\{e_1, \ldots,e_{i-1},e_{i+1},\ldots,e_n \} \quad \text{and} \quad P_i':= a_i \cdot e_i + \mathrm{span}\{e_1, \ldots,e_{i-1},e_{i+1},\ldots,e_n \}.$$ Let $\tau_i$ be the reflection in the hyperplanes $P_i'$. We identify all points that are mapped to each other by the concatenation of some combination of the $\tau_i$, i.e. 
\begin{align}
\label{equivclass}
   v \sim w \ \Leftrightarrow\, \exists I \subseteq [n], \ I=\{i_1,\ldots,i_k\}: \tau_{i_1}\circ \ldots\circ \tau_{i_k} (v) = w.
\end{align}
This yields that a point $v=(v_1,\ldots,v_n)\in \cT$ is identified with all points from the set
\begin{align*}
   Q(v)= \left\{ \begin{pmatrix} y_1 \\ \vdots \\ y_n\end{pmatrix} \in 2\cT \,\middle\vert\, y_i \in \{v_i, 2a_i- v_i\}\right\},
\end{align*}
see \Cref{Fig:TransStruc} for an illustration. 
Points in the interior of $\cT$ have $2^n$ copies in $2\cT$. Points on a $k$-dimensional border only have $2^k$ copies.
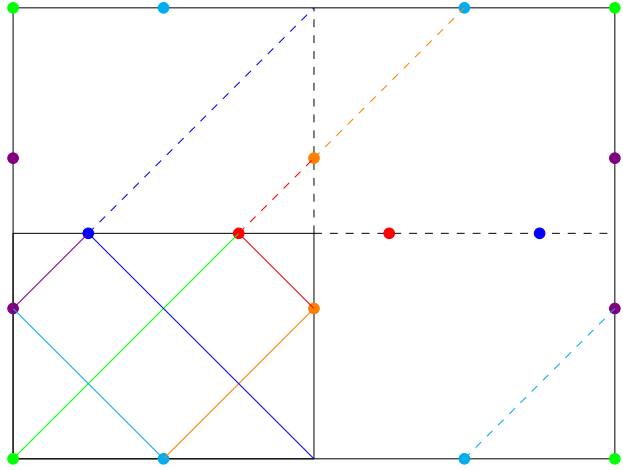
\begin{figure}[h]
\centering
\begin{tikzpicture}
    \draw[line width=0.1mm] (0,0) rectangle (4,3);
    \draw[line width=0.1mm] (0,0) rectangle (8,6);
    \draw[dashed, line width=0.1mm] (4,3) -- (8,3);
    \draw[dashed, line width=0.1mm] (4,3) -- (4,6);
    \filldraw[red] (3,3) circle (2pt);
    \filldraw[red] (5,3) circle (2pt);
    \filldraw[orange] (4,2) circle (2pt);
    \filldraw[orange] (4,4) circle (2pt);
    \filldraw[green] (0,0) circle (2pt);
    \filldraw[green] (8,0) circle (2pt);
    \filldraw[green] (0,6) circle (2pt);
    \filldraw[green] (8,6) circle (2pt);
    \filldraw[blue] (1,3) circle (2pt);
    \filldraw[blue] (7,3) circle (2pt);
    \filldraw[cyan] (6,0) circle (2pt);
    \filldraw[cyan] (2,0) circle (2pt);
    \filldraw[cyan] (6,6) circle (2pt);
    \filldraw[cyan] (2,6) circle (2pt);
    \filldraw[violet] (0,2) circle (2pt);
    \filldraw[violet] (0,4) circle (2pt);
    \filldraw[violet] (8,2) circle (2pt);
    \filldraw[violet] (8,4) circle (2pt); 
    \draw[green] (0,0) -- (3,3);
    \draw[red] (3,3) -- (4,2);
    \draw[orange] (4,2)-- (2,0);
    \draw[cyan] (2,0) -- (0,2);
    \draw[violet] (0,2) -- (1,3); 
    \draw[blue] (1,3) -- (4,0);
    \draw[red, dashed] (3,3) -- (4,4);
    \draw[orange, dashed] (4,4) -- (6,6);
    \draw[cyan, dashed] (6,0) -- (8,2);
    \draw[blue, dashed] (1,3) -- (4,6);
\end{tikzpicture}
\caption{The points with the same color are identified via the reflections. The old trajectory on $\cT$ is displayed with a solid line. 
The dashed line is the new trajectory on $2\cT$. Corresponding line segments in both trajectories are colored with the same color. Note that the green and purple line segments are coinciding.}
\label{Fig:TransStruc}
\end{figure}
\noindent In $2\cT$ we glue opposite sides together and obtain an $n$-dimensional torus. Instead of reflecting the direction of the trajectory on the border we continue the trajectory on the point opposite to the bouncing point, see also \Cref{Fig:TransStruc}. The advantage of this setting is that the direction of the trajectory is not changed at the border. This implies that each point can be crossed in exactly one way, therefore we now can count the number of identified points that are crossed to get $m_\cT(v)$. The drawback of this new setup is that the trajectory now is not connected (when displayed in $\mathbb{R}^n$), see \Cref{Fig:TransStruc}.\par\medskip

\noindent In the next step of our analysis we will even forget the gluing of the opposing sides and consider the whole space $\RR^n$. We identify points in $\RR^n$ with points in $2\cT$ via the function $\psi:\RR^n \to 2\cT$, given by
\begin{equation*}
    \psi(\begin{pmatrix} v_1\\ \vdots \\ v_n \end{pmatrix})= \begin{pmatrix} v_1 \bmod 2a_1 \\ \vdots \\ v_n \bmod 2a_n \end{pmatrix}.
\end{equation*}
If $v$ is of the form $(x,\ldots,x)$ we also use the following notation
\begin{equation*}
    \psi(x)= \begin{pmatrix} x \bmod 2a_1 \\ \vdots \\ x \bmod 2a_n
\end{pmatrix}.
\end{equation*}
By concatenation of the two identifications, via \eqref{equivclass} and $\psi$, we get another identification of points, such that $v=(v_1,\ldots,v_n)\in \cT$ is identified with all points from the set
\begin{align*}
   R(v)= \left\{ \begin{pmatrix} y_1 \\ \vdots \\ y_n\end{pmatrix} \in \RR^n\,\middle\vert\, y_i \equiv v_i \bmod 2a_i\ \ \text{or} \ \ y_i \equiv - v_i \bmod 2a_i \right\},
\end{align*}
Now the entire trajectory in $\mathbb{R}^n$ is a straight line in direction $d=(1,\ldots,1)$ which starts at the origin and, by \Cref{lem:LengthTrajectory}, ends at $\sqrt{n} \cdot (\ell,\ldots,\ell)$.

\section{Determining crossing numbers via CSPs: Proof of Theorem~\ref{Thm:Main}}\label{sec:mainProof}

Recall that for the proof of Theorem~\ref{Thm:Main} the side lengths $a_i$ can be assumed to be integers. We do not assume them to be additionally pairwise coprime.\par\medskip

We use the language of \emph{constraint satisfaction problems} (CSP). 
	Let $\P$ be the family of all subsets of $\{0,1\} \times \{0,1\}$. The elements of $\P$ are referred to as \emph{constraints}.
	A (boolean) CSP on a finite set $V\subseteq \mathbb{N}$ is a pair $(V,\phi)$, where $\phi: \binom{V}{2} \rightarrow \P$ is a function assigning a constraint for each pair of vertices. 
	An \emph{assignment} on $V$ is a function $g: V \rightarrow \{0,1\}$ which assigns for every $v\in V$ an integer from $\{0,1\}$. An assignment $g: V \rightarrow \{0,1\}$ is \emph{satisfying} for $(V,\phi)$ if 
	$(g(a),g(b)) \notin \phi(\{a,b\})$ for any pair $a,b\in V$ such that $a<b$.
	For a CSP $G = (V,\phi)$, denote by $A(G)$ the set of satisfying assignments for $G$. Define 
	\begin{align*}
		\mathcal{C}:=\bigg\{\emptyset,\Big\{(0,0),(1,1)\Big\},\Big\{(0,1),(1,0)\Big\},\Big\{(0,0),(0,1),(1,0),(1,1)\Big\}\bigg\}\subseteq \P,
	\end{align*}
and, given $v\in \mathbb{N}^n \cap \cT$, a function $\phi_v: \binom{[n]}{2} \rightarrow \C$ by 
\begin{align*}
\phi_v(ij):=\begin{cases}  \emptyset&,  \text{if}\ v_i\equiv -v_j \ \text{and} \ v_i\equiv v_j \bmod \gcd(2a_i,2a_j), \\
 \{(0,1),(1,0)\}&,  \text{if}\ v_i\equiv v_j \ \text{and} \ v_i\not\equiv -v_j \bmod \gcd(2a_i,2a_j), \\
 \{(0,0),(1,1)\}&,   \text{if}\ v_i\equiv-v_j \ \text{and} \ v_i \not\equiv v_j \bmod \gcd(2a_i,2a_j), \\
 \{(0,0),(0,1),(1,0),(1,1)\}&,   \text{if}\ v_i \not\equiv -v_j \ \text{and} \ v_i \not\equiv v_j \bmod \gcd(2a_i,2a_j). 
\end{cases}
\end{align*}
Let $v, x, a \in \ZZ^n$. We write $v \equiv x \bmod a$ iff $v_i \equiv x_i \bmod a_i$ for all $i \in [n]$. 
Given $v\in \mathbb{N}^n \cap \cT$, define $J(v):=\{i\in [n]: \ v_i \equiv -v_i \bmod 2a_i \}$.
The following lemma sets the crossing number in relation to the number of satisfying assignments of a CSP.
\begin{lemma}
\label{connectionCSPc}
Let $v\in \mathbb{N}^n \cap \cT$ be a non-corner and $G=([n],\phi_v)$. Then,  
\begin{align*}
m_{\cT}(v)=  \frac{1}{2^{|J(v)|+1}}|A(G)|.
\end{align*}
\end{lemma}
\begin{proof}
Let $a:= (a_1,\ldots,a_n)$ and $\ell=\lcm(a_1,a_2,\ldots,a_n)$. 
The trajectory in $\cT$ passes through $v$ exactly if the trajectory in $\RR^n$ passes through a point which is in $R(v)$, i.e. if the following system of equations has one solution $x\in \{0,1,\ldots,\ell\}$ for all $i \in [n]$:
\begin{equation}
    \label{system}
    x \equiv v_i \bmod 2a_i \quad \text{or} \quad x \equiv - v_i \bmod 2a_i.
\end{equation}
Note that if $v$ satisfies $v_i \equiv -v_i \bmod 2a_i$ for some $i$, then there is only one equation in the $i$-th condition of \eqref{system}. Observe that iff $x\in \{1,\ldots,\ell-1\}$ satisfies \eqref{system}, then also $2\ell-x\in \{\ell+1,\ldots,2\ell-1\}$ satisfies \eqref{system}. Also $x=0$ and $x=\ell$ do not 
 satisfy \eqref{system}, because $\psi(0)=0$ and $\psi(\ell)\equiv 0 \bmod a$.  Therefore,
\begin{align}
\label{count}
\nonumber
m_{\cT}(v)&=\lvert\big\{x\in \{0,1,\ldots,\ell-1\}: \ x \equiv v_i \bmod 2a_i \ \text{or} \ x \equiv - v_i \bmod 2a_i \quad \text{for all } i\in [n] \big\}\rvert\\
&= \frac{1}{2}\lvert\big\{x\in \{0,1,\ldots,2\ell-1\}:  \ x \equiv v_i \bmod 2a_i \ \text{or} \ x \equiv - v_i \bmod 2a_i \quad \text{for all } i\in [n] \big\}\rvert.
\end{align}
By the Chinese Remainder Theorem, \Cref{Thm:ChineseRemainder}, iff $g: [n] \rightarrow \{0,1\}$ is a satisfying assignment for $G$, then there exists a unique solution $x\in \{0,1,\ldots,2\ell-1\}$ to the system of equations
\begin{align}
\label{sysequations}
x \equiv (-1)^{g(i)}v_i \bmod 2a_i \quad \text{for } i\in [n]. 
\end{align}
Two satisfying assignments $g,g'$ for $G$ correspond to the same set of equations in \eqref{sysequations} iff $g(i)=g'(i)$ for all $i\in[n]\setminus J(v)$. We conclude that 
\begin{align*}
m_{\cT}(v)&= \frac{1}{2}\lvert\big\{x\in \{0,1,\ldots,2\ell-1\}:  \ x \equiv v_i \bmod 2a_i \ \text{or} \ x \equiv - v_i \bmod 2a_i \quad \text{for all } i\in [n] \big\}\rvert \\
&= \frac{1}{2} \cdot \frac{1}{2^{|J(v)|}}|A(G)|=  \frac{1}{2^{|J(v)|+1}}|A(G)|,
\end{align*}
completing the proof of this lemma.
\end{proof}

\begin{lemma}
		\label{CSPcount}
		Let $n \in \mathbb{N}$, $\phi: \binom{[n]}{2} \rightarrow \mathcal{C}$ and $G = ([n],\phi)$ be a CSP. Then 
\begin{align*}
|A(G)|\in \{0 \}  \cup \{2^k\mid k\geq 1\}.
\end{align*}
\end{lemma}

	\begin{proof}
Construct an auxiliary graph $H$ from the CSP $G = ([n],\phi)$ as follows. The vertex set of $H$ is $V(H)=[n]$ and the edge set $E(H)=\{e\in \binom{[n]}{2}: \ \phi(e)\neq \emptyset\}$. Let $C_1,\ldots, C_k \subseteq [n]$ be the connected components of $H$. For $i\in [k]$, let $$G[C_i]:=\left(C_i,\restr{\phi}{\binom{C_i}{2}}\right)$$ be the \emph{induced CSP} of $G$ on $C_i$. Then $g\in A(G)$ iff $\restr{g}{C_i}\in A(G[C_i])$ for all $i\in [k]$, simply because for every pair $ab$ with $a<b, a\in C_i, b\in C_j$ for $i\neq j$  we have $\phi(ab)=\emptyset$. We conclude
\begin{align*}
|A(G)|=\prod_{i=1}^k |A(G[C_i])|.
\end{align*}
Next, we show that $|A(G[C_i])|\in \{0,2\}$ for $i\in [k]$, completing the proof of this lemma. 
Let $i\in [k]$. Then, since $H[C_i]$ is connected, there exists a spanning tree $T\subseteq H[C_i]$. Let $C_i=\{w_1,w_2,\ldots, w_{|C_i|}\}$ such that $T[W_j]$, where $W_j:=\{w_1,\ldots,w_j\}$, is a tree. We prove inductively, for $j\in [|C_i|]$ that there exists at most one satisfying assignment $g$ of $G[W_j]$ such that $g(w_1)=0$. For $j=1$, this is trivial. Now, let $j\in [|C_i|]$ such that there exists at most one satisfying assignment $g$ of $G[W_j]$ such that $g(w_1)=0$. If there does not exist a satisfying assignment of $G[W_j]$ such that $g(w_1)=0$, then in particular there does not exist a satisfying assignment of $G[W_{j+1}]$ such that $g(w_1)=0$. Let $e=xw_{j+1}\in E(T)$ be the unique edge such that $x\in W_j$. Now, assume there exists a satisfying assignment $g$ of $G[W_j]$ such that $g(w_1)=0$. If $g'$ is a satisfying assignment of $G[W_{j+1}]$, then $g'(w_a)=g(w_a)$ for all $a \in [j]$ and
\begin{align*}
g'(w_{j+1})=\begin{cases}
0, &\text{if }  g(x)=1   \text{ and } \phi(e)=\Big\{(0,0),(1,1)\Big\}, \\
1, &\text{if }  g(x)=1   \text{ and } \phi(e)=\Big\{(0,1),(1,0)\Big\},\\
1, &\text{if }  g(x)=0   \text{ and } \phi(e)=\Big\{(0,0),(1,1)\Big\}, \\
0, &\text{if }  g(x)=0   \text{ and } \phi(e)=\Big\{(0,1),(1,0)\Big\}.\\
\end{cases}
\end{align*}
Thus, there is at most one satisfying assignment $g'$ of $G[W_{j+1}]$ such that $g'(w_1)=0$. Inductively, we conclude that there is at most one satisfying assignment of $G[C_i]$ such that $g(w_1)=0$. Similarly, there is at most one satisfying assignment of $G[C_i]$ such that $g(w_1)=1$. Thus, $|A(G[C_i])|\leq 2$. Further, if $g$ is a satisfying assignment of $G[C_i]$, then $g': C_i \rightarrow \{0,1\}$ defined by $g'(j)=1-g(j)$ for $j\in C_i$ is also a satisfying assignment of $G[C_i]$. We conclude $|A(G[C_i])|\in \{0,2\}$.
\end{proof}

\begin{proof}[Proof of \Cref{Thm:Main}]
If $v$ is a corner, then $m_{\cT}(v)=0$ or $m_{\cT}(v)=1$ simply because the trajectory starts at the origin and its run is over when it hits a corner the first time. If $v$ is a non-corner, then Theorem~\ref{Thm:Main} follows by combining Lemma~\ref{connectionCSPc} with Lemma~\ref{CSPcount}.
\end{proof}

\section{Concluding Remarks}
The reader may ask how the crossing number varies as the set-up of the board changes. We note that for different rational starting slopes $d$ and for parallelepipeds instead of boxes, the result is not changed, i.e. the crossing number of a given point is still $0$ or a power of $2$. This observation follows by considering a linear transformation that transforms the new set-up into the old one without changing the crossing number.

\begin{lemma}\label{lem:Extension}
    Let $\Phi: \RR^n \to \RR^n$ be an isomorphism, $\cT \subset \RR^n$ and $d$ the starting direction of the trajectory. Then $m_{\cT}(v) = m_{\Phi(\cT)}(\Phi(v))$, whereby $\Phi(d)$ is the starting direction on $\Phi(\cT)$. 
    \begin{proof}
        We consider the trajectory as an object consisting of multiple line segments. All line segments intersecting in $v$ are mapped to line segments intersecting in $\Phi(v)$, so $m_{\cT}(v) \leq m_{\Phi(\cT)}(\Phi(v))$ holds. Since $\Phi$ is an isomorphism it is invertible and therefore by the same argument 
        
        \begin{equation}
        \label{isomorph}
        m_{\cT}(v) \leq m_{\Phi(\cT)}(\Phi(v)) \leq m_{\Phi^{-1}(\Phi(\cT))}(\Phi^{-1}(\Phi(v)))=m_{\cT}(v),
        \end{equation}
        Thus the inequalities in \eqref{isomorph} hold with equality, as desired.
    \end{proof}
\end{lemma}
\noindent To see that \Cref{Thm:Main} also holds in a box $A$ with a different starting direction $d\in \mathbb{Q}^n$ one maps this box $A$ to another box $\Phi(A)$ in such a way that the new starting direction is $\Phi(d)=(1,1,\ldots,1)$ again.\par
\noindent Additionally, if one considers the domain to be parallelepiped instead of a box, one can map it to a box via an isomorphism.

\section*{Acknowledgments}
We thank the members of Kaffeerunde IAG at KIT as well as Cameron Gates Rudd for stimulating discussions on the topic. 

\bibliographystyle{abbrv}
\bibliography{billiards_arxiv_3_06}
\end{document}